\documentclass[a4paper,10pt]{amsart}

\usepackage{graphicx}
\usepackage{amscd}
\usepackage{stmaryrd}
\usepackage{amsthm}
\usepackage{amssymb}
\usepackage{amsfonts}
\usepackage{amsmath}
\usepackage{mathrsfs}
\usepackage{hyperref}
\usepackage[all]{xy}
\usepackage[english]{babel}
\usepackage[utf8]{inputenc}

\usepackage[usenames]{color}

\newcommand{\bbF}{\mathbb{F}}
\newcommand{\bbQ}{\mathbb{Q}}

\newcommand{\bbN}{\mathbb{N}}

\newcommand{\bbP}{\mathbb{P}}
\newcommand{\scF}{\mathscr{F}}
\newcommand{\clN}{\mathcal{N}}
\newcommand{\sep}[1]{\bar{#1}^\mathrm{sep}}
\newcommand{\Din}{\rotatebox{270}{$\in$}}
\newcommand{\Dnotin}{\rotatebox{270}{$\notin$}}
\newcommand{\generated}[1]{\langle #1 \rangle}
\newcommand{\pdiff}[1]{\frac{\partial}{\partial#1}}

\theoremstyle{plain}
 \newtheorem{prop}[subsection]{Proposition}

\begin{document}

\title[Field Generated by Newton Polynomials]{Generation of the Symmetric Field by Newton
Polynomials in prime Characteristic}
\author[M. Monge]{Maurizio Monge}
\address{Scuola Normale Superiore di Pisa - Piazza dei Cavalieri, 7 - 56126 Pisa}
\email{maurizio.monge@sns.it}
\date{\today}

\keywords{Netwon polynomial, symmetric function, Schur function, irreducibility, function field,
  prime characteristic}
\subjclass[2010]{05E05, 12F20, 12E10, 12E20}

\begin{abstract}
Let $N_m = x^m + y^m$ be the $m$-th Newton polynomial in two variables, for $m \geq 1$.
Dvornicich and Zannier proved that in characteristic zero three Newton polynomials $N_a, N_b,
N_c$ are always sufficient to generate the symmetric field in $x$ and $y$, provided that $a,b,c$ are
distinct positive integers such that $(a,b,c)=1$. In the present paper we prove that in case of
prime characteristic $p$ the result still holds, if we assume additionally that $a,b,c,a-b,a-c,b-c$
are prime with $p$. We also provide a counterexample in the case where one of the hypotheses
is missing.

The result follows from the study of the factorization of a generalized Vandermonde determinant in
three variables, that under general hypotheses factors as the product of a trivial
Vandermonde factor and an irreducible factor. On the other side, the counterexample is connected to
certain cases where the Schur polynomials factor as a product of linear factors.
\end{abstract}

\maketitle

\section{Introduction}

Let $F = k(x,y)$ be the function field generated over a field $k$ by the algebraically independent
transcendentals $x, y$, and let $S$ be the subfield of symmetric functions. Let $N_m$ be the $m$-th
Newton polynomial (or power sum) in $x$ and $y$
\begin{equation*}
 N_m \ =\  x^m + y^m,\qquad\text{for }m\geq 1.
\end{equation*}
Note that if the characteristic is $p$, then $N_{p^km} = N_m^{p^k}$, for each $k\in\bbN$.

We will also call $\clN_{a,b}$ (resp. $\clN_{a,b,c}$) the subfield of $F$ generated by $N_a,N_b$
(resp. $N_a,N_b,N_c$) over $k$, i.e.
\begin{equation*}
 \clN_{a,b} \ =\  k(N_a,N_b),\qquad \clN_{a,b,c} \ =\  k(N_a,N_b,N_c).
\end{equation*}

In \cite{MeadStein1998} Mead and Stein calculated the degree of the extension $S/\clN_{a,b}$ in
characteristic zero, and conjectured that $S = \clN_{a,b,c}$ (i.e. $N_a,N_b,N_c$ generate the whole
symmetric field) whenever $a,b,c$ are distinct integers such that $(a,b,c)=1$, also providing
evidence for their conjecture. This conjecture was finally settled in \cite{DvornicichZannier2003}
by Dvornicich and Zannier, by computing the Galois group of a polynomial connected to a
fundamental determinantal equation via the Riemann Existence Theorem. The solution of the
conjecture then followed after they proved that such Galois group must be the full
symmetric group, and considering the action of the Galois group on a system of equations
connected to the problem.

These topological methods do not seem to admit an immediate generalization to the prime
characteristic case. However we will show that the same result also follows from the
irreducibility of the main factor of the fundamental determinantal equation, and that in many cases
such irreducibility can be proved by elementary methods.

It should also be noted that it is not possibile to expect the conjecture to hold
in prime characteristic without any additional hypothesis, since whenever $a,b$ are distinct
positive integers such that $(a,b)=1$, then $a,pa,b$ are coprime integers and $N_{pa}=N_a^p$, so
$\clN_{a,pa,b} = \clN_{a,b}$ (that can be easily seen to be $\neq S$, in general).
It is not enough just to request that $a, b, c$ are all prime with $p$: in the last
section we show that there exist triples of coprime integers $a,b,c$, all prime with $p$, such
that $\clN_{a,b,c}$ is strictly contained in $S$. The degree of such non-trivial extension
can be computed explicitly, and we will also exhibit a formula for this degree for a family of
triples $a,b,c$ such that the differences are not all prime with $p$.

\subsection{Acknowledgements} We wish to thank Roberto Dvornicich for reading the first draft of
this paper and for offering kind support and advice during our work.

\section{Preliminary results}

We will now prove that the solution of our problem only depends on the characteristic of the field
of constants $k$. This allows us to replace $k$ with any other field with the same characteristic,
provided that $x,y$ are still algebraically independent.

For any field $L$, let $\clN_a^L$ be the field generated over $L$ by a collection of Newton
polynomials $N_{a_1}, \dots, N_{a_s}$ (actually we will always be in the case of $s=2,3$).
Similarly, let $F^L$ (resp. $S^L$) be the field of functions (resp. symmetric functions) in $x,y$
over $L$. Then we have:

\begin{prop}
Let $K$ be a field such that $x,y$ are algebraically independent over $K$, and let $k$
be a subfield. Then
\begin{equation*}
 [S^K : N_a^K] \ =\ [S^k : N_a^k],
\end{equation*}
intending that whenever one of the two degrees is finite, then the other one
is finite too and the two degrees are equal.
\end{prop}

\begin{proof}
In fact, we can consider the following diagram:
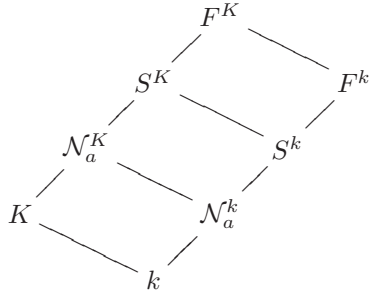
\begin{displaymath}
\xymatrix@!0@=25pt{
&    &    &   F^K\ar@{-}[ld]\ar@{-}[rrd]   \\
&    &  S^K\ar@{-}[ld]\ar@{-}[rrd]  &       &     & F^k\ar@{-}[ld]  \\
&  \clN^K_a \ar@{-}[ld]\ar@{-}[rrd]  &    &       &     S^k\ar@{-}[ld]                          \\
K\ar@{-}[rrd] &    &    & \clN^k_a \ar@{-}[ld] \\
&    & k }
\end{displaymath}

Being $x,y$ algebraically independent over $K$, it follows that $K$ and $F^k$ are linearly disjoint
over $k$ (see \cite[Prop. 3.3, chap. VIII]{Lang2002}). Consequently we have that $\clN^K_a$
and $S^k$ are linearly disjoint over $\clN^k_a$ as well (see \cite[Prop. 3.1, chap.
VIII]{Lang2002}), implying the equality whenever $S^k/\clN^k_a$ or $S^K/\clN^K_a$ is an algebraic
extension of finite degree.
\end{proof}

Note that this proposition allows us to replace the field $K$ with any other field $L$ with the
same characteristic, as they both contain the same prime field ($\bbQ$ or $\bbF_p$).

Another simple but crucial observation is that we just need to calculate the degree
$[S:\clN_a]$ when $\clN_a$ is the field generated by a collection of Newton polynomials
$N_{a_1},\dots,N_{a_s}$ satisfying $(a_1,\dots,a_s) = 1$. In fact, if the $a_1,\dots,a_s$
have a non-trivial gcd, $g$ say, we obtain a field that is contained in the symmetric field
in $x^g,y^g$. If we call $\clN_{a/g}$ the field generated by $N_{a_1/g},\dots,N_{a_s/g}$, and call
$S^{(g)}$ the symmetric field in $x^g,y^g$, we have that
\begin{equation*}
 [S : \clN_a] = [S : S^{(g)}] [S^{(g)} : \clN_a] = g^2 [S : \clN_{a/g}].
\end{equation*}
Consequently, we can always assume $(a_1,\dots,a_s) = 1$ with no loss of generality.

\subsection{Field generated by two polynomials} We now extend to the prime characteristic case the
results obtained by Mead and Stein in \cite{MeadStein1998}. We have the following

\begin{prop}
Let $a > b$ be coprime integers. If $k$ has positive characteristic $p$, let's assume $a,b$
prime with $p$. Then the extension $S/\clN_{a,b}$ is a separable algebraic extension of degree
\begin{equation*}
 [S : \clN_{a,b}] \ =\
    \left\{\begin{array}{cl}
        \frac{ab}{2} & \text{if }ab\text{ is even,}\\
        \frac{(a-1)b}{2} & \text{if }ab\text{ is odd.}
    \end{array}\right.
\end{equation*}
\end{prop}
\begin{proof}
We can immediately see that the the extension is a finite separable algebraic extension, because
the Jacobian (see \cite[Sec. 5, chap. VIII]{Lang2002}) of the algebraic map $(x,y) \mapsto
(x^a+y^a,x^b+y^b)$ is
\begin{equation*}
  \det\begin{pmatrix} ax^{a-1} & ay^{a-1} \\ bx^{b-1} & by^{b-1} \end{pmatrix}
  \ =\ ab(x^{a-1}y^{b-1}-x^{b-1}y^{a-1}),
\end{equation*}
that is $\neq 0$ since we assumed $a,b$ to be prime with the characteristic, or the characteristic
to be zero. Note for future reference that we did not need to assume $a,b$ to be coprime to achieve
this.

To calculate the degree of the extension we will calculate the degree of $F/\clN_{a,b}$, the degree
of $S/\clN_{a,b}$ will be precisely half of it. Observe that when we add $x$ to the field
$\clN_{a,b}$ we get
\begin{equation*}
 y^a \ =\ N_a - x^a,\qquad y^b \ =\  N_b - x^b,
\end{equation*}
and consequently $y \in \clN_{a,b}(x)$ being $a,b$ relatively prime. This proves that $x$ is a
primitive element for $F$, i.e. $F = \clN_{a,b}(x)$, so we have to calculate the degree of $x$ over
$\clN_{a,b}$.

But $x$ is a solution of the polynomial
\begin{equation*}
 f(X) \ =\ (N_a - X^a)^b - (N_b - X^b)^a \in \clN_{a,b}[X],
\end{equation*}
that is homogeneous of weight $ab$, if we assign weight $1$ to $X$ and $a,b$ to $N_a,N_b$
respectively. Furthermore, $N_a,N_b$ must be algebraically independent over $k$, since $F$ has
transcendence degree $2$ over $k$, and $F/\clN_{a,b}$ is algebraic.

The constant term of $f(X)$ is $N_a^b - N_b^a$, and it is irreducible. In fact it is homogeneous of
weight $ab$, and a factor should have weight multiple of both $a$ and $b$, and thus $ab$ being $a,b$
relatively prime. Consequently $f(X)$ is irreducible too being homogeneous, and its degree in $X$ is
$ab$ when one of $a,b$ is even (and in this case $p\neq 2$, since $(ab,p)=1$), or $(a-1)b$
otherwise.
\end{proof}

\section{Case with some of \texorpdfstring{$a,b,c$}{a,b,c} divisible by \texorpdfstring{$p$}{p}}

In this section we will work in characteristic $p$, assuming the base field to be $\bar{\bbF} =
\bar{\bbF}_p$ for convenience. We will see that we are not actually losing much assuming all of
$a,b,c$ to be prime with $p$. In fact, we have

\begin{prop}\label{newt:prop1}
Suppose that at least two of $a,b,c$ are divisible by $p$. Then $\clN_{a,b,c}$ cannot be $S$.
\end{prop}
\begin{proof}
Assume $a,b$ to be divisible by $p$. Then $\clN_{a,b,pc}$ is contained in $S^{(p)}$, the
symmetric function field in $x^p,y^p$. But $\clN_{a,b,c}$ has degree at most $p$ over
$\clN_{a,b,pc}$, since it is generated by $N_c$, that satisfies $X^p-N_{pc}$. On the other end, the
degree $[S : S^{(p)}]$ is $p^2$, so $N_{a,b,c}$ cannot be equal to $S$.
\end{proof}

On the other hand, the following proposition shows that the case with only one among
$a,b,c$ divisible by $p$ can be reduced to the case where they are all prime with $p$.

\begin{prop}\label{newt:prop2}
Suppose $a,b,c$ all prime with $p$. Then for all $k \geq 1$ we have $\clN_{a,b,c} =
\clN_{a,b,p^kc}$.
\end{prop}
\begin{proof}
The extension $\clN_{a,b,c} / \clN_{a,b,p^kc}$ is purely inseparable, being generated by $N_c$
that satisfies the purely inseparable equation
\begin{equation*}
 X^{p^k} - N_{p^kc} = 0.
\end{equation*}
But this extension is contained in the extension $S/\clN_{a,b,p^kc}$, that is separable (since as we
have seen in the proof of Prop. \ref{newt:prop1}, $S/\clN_{a,b}$ is separable). Thus being both
separable and purely inseparable the extension $\clN_{a,b,c} / \clN_{a,b,p^kc}$ must be trivial.
\end{proof}

\section{The main result}

Most of this section is dedicated to proving the following

\begin{prop} \label{newt:prop3}
Let $a>b>c$ be relatively prime positive integers, and suppose that $a,b,c,a-c,a-b,b-c$ are
prime to the characteristic $p$. Then we have that
\begin{equation*}
 \clN_{a,b,c} = S.
\end{equation*}
\end{prop}

\begin{proof}

We will argue by contradiction, assuming the degree of $F/\clN_{a,b,c}$ to be $\gneqq 2$.
Let $\sep{F}$ be a separable algebraic closure of the rational functions $F$, and suppose
that there exists $z,w \in \sep{F}$ different from $x,y$ such that
\begin{equation} \label{newt:eq1}
 x^m+y^m \ =\  z^m+w^m,\qquad\text{for }m=a,b,c.
\end{equation}
Since $x,y$ are separable over $\clN_{a,b,c}$ we can restrict our attention to the
separable closure $\sep{F}$.

It is easy to see that there cannot be two of $x,y,z,w$ with constant ratio: in fact $x,y$ are
algebraically independent, and the same must be true for $z,w$, since  $k(z,w) \supseteq
\clN_{a,b,c}$, and $\clN_{a,b,c}$ has transcendence degree $2$. Now suppose that $z=\mu x$, with
$\mu \in k$. Then replacing $z$ with $\mu x$ and eliminating $w$ from \eqref{newt:eq1} we have
\begin{align*}
 \big((1-\mu^a)x^a + y^a\big)^b = \big((1-\mu^b)x^b + y^b\big)^a,\\
 \big((1-\mu^a)x^a + y^a\big)^c = \big((1-\mu^c)x^c + y^c\big)^a,\\
\end{align*}
that are relations between $x$ and $y$, that are algebraically independent. Consequently they must
be trivial, and considering the coefficients of $x^b,x^c,x^a$ we deduce that this can happen if and
only if $\mu^a = \mu^b = \mu^c = 1$, i.e. $\mu = 1$ being $(a,b,c) = 1$. But in this case $x=z$ and
$y=w$, and $x,y$ are not different from $z,w$.

To proceed let's extend to $\sep{F}$ the standard derivation $\partial/\partial z$ on the field
$\bar{\bbF}(z,w)$ (as we said before, $z,w$ are algebraically independent), that we will indicate
with a prime. Taking the derivative of \eqref{newt:eq1} we get the non trivial relations
\begin{equation} 
 x^{m-1}x'+y^{m-1}y' \ =\  z^{m-1},\qquad\text{for }m=a,b,c,
\end{equation}
since we required $a,b,c$ to be all prime with $p$. This system of equations can also be written as
\begin{equation*}
 \begin{pmatrix}
  x^{a-c} & y^{a-c} & z^{a-c} \\
  x^{b-c} & y^{b-c} & z^{b-c} \\
  1 & 1 & 1
 \end{pmatrix}
   \cdot
 \begin{pmatrix}
  x'\cdot x^{c-1} \\ y'\cdot y^{c-1} \\ z^{c-1}
 \end{pmatrix}
  \ =\ 
 \begin{pmatrix}
  0 \\ 0 \\ 0
 \end{pmatrix}.
\end{equation*}
This last equation shows that $x,y,z$ must be solutions of the determinantal polynomial $R(X,Y,Z)$
defined as
\begin{gather} \label{newt:eq2}
R(X,Y,Z) \ =\ 
\det
 \begin{pmatrix}
  X^{a-c} & Y^{a-c} & Z^{a-c}\\
  X^{b-c} & Y^{b-c} & Z^{b-c}\\
  1 & 1 & 1 \\
 \end{pmatrix} \ = \\
  Z^A(X^B-Y^B)-Z^B(X^A-Y^A)+X^BY^B(X^{A-B}-Y^{A-B}), \notag
\end{gather}
where we have put $A=a-c$, $B=b-c$, and that we will see as a polynomial in $Z$ with coefficients in
$\bar{\bbF}[X,Y]$.

Let $V(X,Y,Z) = (X-Y)(Z-X)(Z-Y)$, the Vandermonde determinant in $X,Y,Z$. If we let
$d=(A,B)$, then $R(X,Y,Z)$ is clearly divisible by $V(X^d,Y^d,Z^d)$, that is not zero on
$(x,y,z)$, since no two of $x,y,z$ have constant ratio.

Thus the quotient
\begin{equation*}
 T(X,Y,Z) \ =\  T_{A,B}(X,Y,Z) \ =\  \frac{R(X,Y,Z)}{V(X^d,Y^d,Z^d)},
\end{equation*}
that we are going to show to be irreducible, must vanish on $(x,y,z)$. Let's observe that
$T(X,Y,Z)$ is symmetric, it is also the Schur polynomial $s_\lambda(X^d,Y^d,Z^d)$ in $X^d,Y^d,Z^d$
associated with the partition $\lambda = (A/d-2,B/d-1,0)$ of the theory of symmetric
functions, following the notation of \cite{Macdonald}.

To prove the irreducibility of $T(X,Y,Z)$ in $\bar{\bbF}[X,Y,Z]$, let's consider the polynomial
\begin{equation*}
 I(X,Y,Z) \ =\  \frac{R(X,Y,Z)}{(X^d-Y^d)}, \\
\end{equation*}
that has an intermediate form between $R(X,Y,Z)$ and $T(X,Y,Z)$, and that we will use to extract 
information about $T(X,Y,Z)$. It can be written as
\begin{equation}\label{newt:eq3}
 Z^A \prod_{\substack{\zeta^B=1\\\zeta^d\neq1}}(X-\zeta Y)
    -Z^B\prod_{\substack{\xi^A=1\\\xi^d\neq1}}(X-\xi Y)
    +X^BY^B\prod_{\substack{\theta^{A-B}=1\\\theta^d\neq1}}(X-\theta Y).
\end{equation}

The $\zeta,\xi,\theta$ appearing in the \eqref{newt:eq3} are respectively the $A$-th,
$B$-th and $(A-B)$-th roots of the unity, with the $d$-th roots removed. They are all different,
since $p$ does not divide $A,B,A-B$, and the greatest common divisor of any two of $A,B,A-B$ is
precisely $d$.

\subsection{Irreducibility of \texorpdfstring{$T(X,Y,Z)$}{T(X,Y,Z)}.}

The strategy we are going to use to prove the irreducibility of $T(X,Y,Z)$ can be seen as a
variation of the Eisenstein's criterion, in a sense that will be specified below.

Let $f(U) = \sum_{i=0}^s f_iU^i \in R[U]$ be a polynomial in $U$ over a commutative unitary ring
$R$ with degree $s\geq r$ for some $r\geq 1$, such that $f_r \notin P$ for some prime ideal $P
\subset R$, $f_j\in P$ for $j<r$ and $f_0 \in P\setminus P^2$
\begin{equation*}
 f(U) \ =\  f_sU^s + \dots + \underset{\substack{\Dnotin\\P}}{f_r}U^r +
        \underset{\substack{\Din\\P}}{f_{r-1}}U^{r-1} + \dots +
        \underset{\substack{\Din\\P}}{f_1}U +
        \underset{\substack{\Din\\P\setminus P^2}}{f_0}.
\end{equation*}
If it can be factored as $f(U) = g(U)h(U)$, we can easily deduce from the factorization modulo $P$
that one of its factors, $g(U) = \sum g_iU^i$ say, must inherit this `signature', and satisfy $g_r
\notin P$, $g_j \in P$ for $j<r$, and $g_0 \in P\setminus P^2$, and in particular its degree is at
least $r$.
If $r$ is equal to $s$, the degree of $f$, this forces $h(U)$ to have degree zero, and we recover
precisely Eisenstein's irreducibility criterion. The polynomials that we are studying do
not satisfy the requirements for Eisenstein's criterion, but this will be compensated by the fact
that they are symmetric.

For convenience, we will call this property of $f(U)$ \emph{signature of length $r$ relative to
 the ideal $P$}, and since we can similarly have such a signature in the first $r$
coefficients of the highest degree terms rather than in the lowest degree
terms
\begin{equation*}
 f(U) \ =\  \underset{\substack{\Din\\P\setminus P^2}}{f_s}U^s +
        \underset{\substack{\Din\\P}}{f_{s-1}}U^{s-1} + \dots +
        \underset{\substack{\Din\\P}}{f_{s-r+1}}U^{s-r+1} +
        \underset{\substack{\Dnotin\\P}}{f_{s-r}}U^{s-r} + \dots +f_0,
\end{equation*}
we will respectively speak about \emph{upper signatures} and
\emph{lower signatures}.

Note that $T(X,Y,Z)$ is primitive in $Z$ (for instance because $I(X,Y,Z)$ is), so
we will just have to show that it cannot split into factors with degree $\geq 1$ in $Z$.

Arguing by contradiction, suppose that $T(X,Y,Z)$ can be factored in $k>1$ irreducible
factors with degree $\geq 1$ in $Z$, $\prod_{i=1}^k G_i(X,Y,Z)$, say. Observing the form of
\begin{equation*}
    I(X,Y,Z) = T(X,Y,Z)(Z^d-X^d)(Z^d-Y^d)
\end{equation*}
that we wrote in \eqref{newt:eq3}, we can see that the terms of degree $<B$ in $Z$ are
divisible by $(X-\theta Y)$ for all $\theta^{A-B}=1,\theta^d\neq1$, that the coefficient of the
constant term in $Z$ is divisible only once, while the coefficient of $Z^B$ is not divisible. Thus,
this polynomial has a lower signature of length $B$ relative to the ideal $P_\theta =
\generated{X-\theta Y}$ for all $\theta^{A-B}=1,\theta^d\neq1$, and similarly it has an upper
signature of length $A-B$ relative to the ideal $Q_\zeta =
\generated{X-\zeta Y}$ for all $\zeta^{B}=1,\zeta^d\neq1$.

This polynomial has both an upper signature and a lower signature, unless either
$d=B$, or $d=A-B$, and these cases will be considered separately.

\subsection{Case 1 (with \texorpdfstring{$B\neq d$}{B =/= d} and \texorpdfstring{$A-B\neq
d$}{A-B =/= d}).}
The irreducible factors of $I(X,Y,Z)$ that inherit an upper (resp. lower) signature must have degree
in $Z$ at least $B$ (resp. $A-B$), and they must be factors of $T(X,Y,Z)$. Since the degree in $Z$
of $T(X,Y,Z)$ is precisely $A-2d$, there must be one `big' factor, $G_1(X,Y,Z)$ say, that inherits 
both an upper and a lower signature, or $T(X,Y,Z)$ would have degree $\geq A$ in $Z$. For the same
reason, this big factor $G_1(X,Y,Z)$ must inherit \emph{all} signatures of $I(X,Y,Z)$
relative to the $P_\theta$ and $Q_\zeta$.

Thus, we have that any product of some of the remaining factors
\begin{equation*}
\prod_{i\in I} G_i(X,Y,Z),\qquad\text{with }I \subseteq \{2,3,\dots,k\},\ I\neq \emptyset,
\end{equation*}
must be monic in $Z$, have constant term of the form $X^rY^s$ (for some $r,s \geq 0$), and in
particular it cannot be symmetric. It follows that $T(X,Y,Z)$ cannot be factored as a non trivial
product of symmetric polynomials, i.e. the action of the symmetric group $S_3$ as permutations of
$X,Y,Z$ on the irreducible factors is transitive.

Such action does not preserve the degree in $Z$, but it preserves the total degree, and the
$G_i(X,Y,Z)$ must have the same total degree. As we have seen, $G_1(X,Y,Z)$ has degree in $Z$ at
least $A-B$, and its leading coefficient is the product of precisely $B-d$ factors in $X,Y$ of the
form $(X-\zeta Y)$. Thus, its total degree is at least $A-d$.

On the other side the total degree of $T(X,Y,Z)$ is precisely $A+B-3d$. Were the number of
factors $\geq 2$, then the total degree should be at least
\begin{equation*}
 2(A-d) \gneqq A+B-3d,
\end{equation*}
since $A>B$ and $d>0$. This contradiction proves that $G_1(X,Y,Z)$ must be the only factor, and that
$T(X,Y,Z)$ is irreducible.

\subsection{Case 2 (with \texorpdfstring{$B = d$}{B=d} or \texorpdfstring{$A-B = d$}{A-B=d}).} Let's
show that the case with $A-B = d$ can be
reduced to the case with $B=d$. Since
\begin{equation*}
\frac{
 \det \begin{pmatrix}
  X^A & Y^A & Z^A\\
  X^{d} & Y^{d} & Z^{d}\\
  1 & 1 & 1 \\
 \end{pmatrix}
}{
 \det \begin{pmatrix}
  X^{2d} & Y^{2d} & Z^{2d}\\
  X^d & Y^d & Z^d\\
  1 & 1 & 1 \\
 \end{pmatrix}
}
\ =\ 
\frac{ (XYZ)^A \cdot
 \det \begin{pmatrix}
  1 & 1 & 1\\
  X^{-A+d} & Y^{-A+d} & Z^{-A+d}\\
  X^{-A} & Y^{-A} & Z^{-A} \\
 \end{pmatrix}
}{ (XYZ)^{2d} \cdot
 \det \begin{pmatrix}
  1 & 1 & 1 \\
  X^{-d} & Y^{-d} & Z^{-d}\\
  X^{-2d} & Y^{-2d} & Z^{-2d}\\
 \end{pmatrix}
},
\end{equation*}
we have that
\begin{equation*}
 T_{A,d}(X,Y,Z) \ =\  (XYZ)^{A-2d} \cdot T_{A,A-d}(X^{-1},Y^{-1},Z^{-1}).
\end{equation*}
Consequently from a factorization of $T_{A,A-d}(X^{-1},Y^{-1},Z^{-1}) \in
\bar{\bbF}[X^{-1},Y^{-1},Z^{-1}]$ we can deduce a factorization of $T_{A,d}(X,Y,Z)$,
distributing factors of $(XYZ)^{A-2d}$ on the factors of $T_{A,A-d}(X^{-1},Y^{-1},Z^{-1})$ to make
all exponents positive. The only case where a non trivial factorization can become a trivial
factorization is when one of the factors of $T_{A,A-d}(X^{-1},Y^{-1},Z^{-1})$ is a non trivial
monomial $X^{-r}Y^{-s}Z^{-t}$ for $r,s,t \geq 0$, but this cannot happen in view of the definition
of $T_{A,A-d}(X,Y,Z)$.

To prove the irreducibility of $T_{A,d}(X,Y,Z)$, we will show that the variety defined in
$\bbP^2(\bar{\bbF})$ is nonsingular. The irreducibility follows immediately, since two irreducible
factors would define two projective varieties with non empty intersection (by the theorem of
B\'ezout, see \cite{Hartshorne1977}), and on a point of this intersection all derivatives of the
product would be zero.

Let's consider first the case with $d=1$, and put for convenience $A=k$ for some integer
$k \geq 2$, and $B=1$. If $k=2$ then $T_{k,1}(X,Y,Z) = 1$, so let's suppose $k > 2$. It's easy to
see with a direct computation, or considering the Jacobi-Trudi identity (see \cite{Macdonald}),
that $T_{k,1}(X,Y,Z)$ is the $(k-2)$-th complete symmetric function, i.e. the sum of all monomials
of degree $k-2$, denoted as $h_{k-2}(X,Y,Z)$ in the notation of \cite{Macdonald}.

We have that
\begin{equation*}
 \left( \pdiff{X}+\pdiff{Y}+\pdiff{Z} \right)T_{k,1}(X,Y,Z) \ =\  k \cdot T_{k-1,1}(X,Y,Z),
\end{equation*}
$k$ times the sum of all monomials of degree $k-3$, since the contribution to the monomial
$X^rY^sZ^t$, for $r,s,t \geq 0, r+s+t = k-3$, is given by
\begin{equation*}
\pdiff{X}X^{r+1}Y^sZ^t + \pdiff{Y}X^rY^{s+1}Z^t + \pdiff{Z}X^rY^sZ^{t+1} \ =\  k\cdot X^rY^sZ^t.
\end{equation*}

Suppose that there exists a point with homogeneous coordinates $(x,y,z)$ satisfying the
system of equations
\begin{equation*}
\left\{\begin{array}{l}
  T_{k,1}(X,Y,Z) \ =\ 0, \\
   \pdiff{X}T_{k,1}(X,Y,Z) \ =\ 0, \\
   \pdiff{Y}T_{k,1}(X,Y,Z) \ =\ 0, \\
   \pdiff{Z}T_{k,1}(X,Y,Z) \ =\ 0.
\end{array}\right.
\end{equation*}

Being $k$ prime to the characteristic $p$, the point $(x,y,z)$ must also be a
solution of
\begin{gather*}
 T_{k,1}(X,Y,Z) - \frac{X}{k} \cdot \left( \pdiff{X}+\pdiff{Y}+\pdiff{Z} \right)T_{k,1}(X,Y,Z)\\
  \ =\  \sum_{i=0}^{k-2}Y^iZ^{k-2-i} \ =\  \prod_{\substack{\phi^{k-1} = 1\\\phi \neq 1}} (Y-\phi
Z).
\end{gather*}

Since we also supposed $k-1$ prime to the characteristic $p$, the coordinates of this point must
satisfy $y = \phi z$ for some $\phi^{k-1}=1, \phi\neq 1$. Repeating the computation
with the other variables we have that any two of $x,y,z$ differ by a $(k-1)$-th root of
the unity different from $1$.

Consequently, this point is of the form $(\phi t,\psi t,t) \in \bbP^2{\bar{\bbF}}$, with $\phi^{k-1}
= \psi^{k-1} = 1$ and $\phi,\psi,1$ all different, and $t\neq 0$. But we have that
\begin{gather*}
 \pdiff{Z} T_{k,1}(X,Y,Z) \ =\  \pdiff{Z}\frac{R(X,Y,Z)}{V(X,Y,Z)} \ = \\
 \frac{kZ^{k-1}(X-Y)-(X^k-Y^k)}{V(X,Y,Z)}-R(X,Y,Z)\frac{\pdiff{Z}V(X,Y,Z)}{V(X,Y,Z)^2},
\end{gather*}
where as usual we called $V(X,Y,Z)$ the Vandermonde determinant, and $R(X,Y,Z)$ the determinantal
polynomial defined in \eqref{newt:eq2} for $A=k,B=1$. Evaluating at $(\phi t, \psi t, t)$, and
taking into account that $R(\phi t,\psi t,t)=0$, we deduce that
\begin{gather*}
 \pdiff{Z} T_{k,1}(X,Y,Z) \bigg|_{(\phi t, \psi t, t)}
 \ =\  (k-1)\frac{t^{k-3}}{(1-\phi)(1-\psi)} \ \neq\  0.
\end{gather*}

Let's now take care of the case with $d \geq 1$, and write $A=kd, B=d$. To prove the
irreducibility of $T_{kd,d}(X,Y,Z) \ =\  T_{k,1}(X^d,Y^d,Z^d)$, we will show
that it defines a nonsingular variety as well. So, let's consider the system of equations

\begin{equation*}
\left\{\begin{array}{l}
  T_{k,1}(X^d,Y^d,Z^d) \ =\ 0, \\
   dX^{d-1}\cdot \pdiff{X}T_{k,1}(X^d,Y^d,Z^d) \ =\ 0, \\
   dY^{d-1}\cdot \pdiff{Y}T_{k,1}(X^d,Y^d,Z^d) \ =\ 0, \\
   dZ^{d-1}\cdot \pdiff{Z}T_{k,1}(X^d,Y^d,Z^d) \ =\ 0.
\end{array}\right.
\end{equation*}
Clearly any point $(x,y,z)$ such that $x,y,z$ are all $\neq 0$ cannot satisfy this system, because
this would imply that $(x^d,y^d,z^d)$ would be singular point for $T(X,Y,Z)$, and this cannot
happen as we have just seen.

So let's suppose that the above equations are satisfied in a point $(x,y,z)$ with $y=0$, say. Such a
point must be a solution of
\begin{equation*}
 T_{k,1}(X^d,0,Z^d)\  =\  \prod_{\substack{\phi^{k-1} = 1\\\phi \neq 1}} (X^d-\phi Z^d),
\end{equation*}
implying that $x^d$ and $z^d$ differ by a factor that is a $(k-1)$-th root of the unity different
from $1$. Consequently $(x^d,y^d,z^d)$ must be of the form $(\phi t,0,t)$ for
$\phi^{k-1}=1,\phi\neq1$ and $t\neq 0$, and all we have to show is that
\begin{equation*}
 \pdiff{Z}T_{k,1}(X,Y,Z)\bigg|_{(\phi t,0,t)}
  \ =\ (k-1)\frac{t^{k-3}}{1-\phi} \ \neq\  0.
\end{equation*}

\subsection{Conclusion.}
We just proved that $T(X,Y,Z)$ is irreducible as a polynomial in $Z$ with coefficients in
$\bar{\bbF}[X,Y,Z]$, and consequently it will also be irreducible in $\bar{\bbF}(X,Y)[Z]$ thanks
to Gauss' lemma, being the ring $\bar{\bbF}[X,Y]$ factorial.

Recall that we are supposing the following equations
\begin{equation} \label{newt:eq4}
 x^m+y^m-z^m \ =\  w^m,\qquad\text{ for }m=a,b,c
\end{equation}
to be satisfied for some $w,z$ different from $x,y$, and that $T(X,Y,Z)$ is a relation satisfied by
$x,y,z$.

Note that $z$ must be transcendental over $\bar{\bbF}(x)$. In fact, suppose that this is not the
case: $y$ is a root of the polynomial $R(x,U,z)$, considered as a polynomial in $U$ over
$\bar{\bbF}(x,z)$, and consequently of $T(x,U,z)$ since no two of $x,y,z$ have constant ratio.
Furthermore, $T(x,U,z)$ cannot vanish identically, since its constant term is a homogeneous
polynomial in $x,z$, i.e. of the form $\prod (x-\theta_i z)$, and $x,z$ do not have constant ratio.

This implies the existence of a non trivial algebraic relation of $y$ over $\bar{\bbF}(x,z)$, and
consequently that $y$ is algebraic over $\bar{\bbF}(x)$, but this is impossibile since we assumed
$x,y$ to be algebraically independent. Let's also note for future reference that $w$ must be
transcendental over $\bar{\bbF}(x)$ as well.

The algebraic independence of $x,z$ allows us to define an isomorphism $\epsilon : \bar{\bbF}(x,y)
\rightarrow \bar{\bbF}(x,z)$ that fixes the constants and such that
\begin{equation*}
 x\mapsto x,\qquad y\mapsto z.
\end{equation*}

Since $z$ is a root of the polynomial $T(x,y,U)$ in $U$, we can extend $\epsilon$ to
$\bar{\bbF}(x,y,z)$ defining the image of $z$ to be any root of
\begin{equation*}
\epsilon T(x,y,U) \ =\  T(\epsilon x,\epsilon y,U) \ =\  T(x,z,U) \ =\  T(x,U,z),
\end{equation*}
being $T(X,Y,Z)$ a symmetric polynomial. In particular we can put $\epsilon(z) = y$.

Let's now extend $\epsilon$ to the algebraic closure of $\bar{\bbF}(x,y,z)$, and let $u =
\epsilon(w)$ (actually we have $w \in \bar{\bbF}(x,y,z)$, but we will not have to use this fact).
Applying $\epsilon$ to the \eqref{newt:eq4} we get
\begin{equation} \label{newt:eq5}
 x^m+z^m-y^m \ =\  u^m,\qquad\text{ for }m=a,b,c.
\end{equation}
Adding together the \eqref{newt:eq4} and \eqref{newt:eq5} we get
\begin{equation} \label{newt:eq6}
 2x^m \ =\  w^m + u^m,\qquad\text{ for }m=a,b,c
\end{equation}
(recall that the hypotheses rule out the case of characteristic $2$). Eliminating
$u$ from the \eqref{newt:eq6} for $m=a,b$ we have
\begin{equation} \label{newt:eq7}
 (2x^a-w^a)^b - (2x^b-w^b)^a \ =\  0.
\end{equation}
This is a non trivial algebraic relation of $w$ over $\bar{\bbF}(x)$, that had been
proved to be transcendental over $\bar{\bbF}(x)$. This contradiction concludes the proof.
\end{proof}

A comment on the hypotheses we required at the beginning of the theorem is needed. Let's restrict
to the case of $a,b,c$ coprime and all prime to $p$, as we are allowed to do thanks to Proposition
\ref{newt:prop2}. Computer experiments show that in many cases where $p$ divides
the differences $a-c,a-b,b-c$ the Newton polynomials $N_a,N_b,N_c$ still generate the
full symmetric field.

A careful analysis of the proof shows that in the \emph{Case 1} of the proof of the irreducibility
of $T(X,Y,Z)$ we did not actually use the fact that $A=a-c$ is prime to $p$ (in the \emph{Case
2} this hypothesis is important and necessary, as we will show with some examples below). We have
omitted this small weakening of the hypothesis to avoid complicating too much the statement.

Furthermore, the conclusive step works flawlessly without $T(X,Y,Z)$ being irreducible, provided
that we know its factors to be \emph{all symmetric polynomials}. It is possible to show examples
where precisely this happens (such as $T_{7,3}(X,Y,Z)$ in characteristic $2$), but it seems
difficult to show this for some class of polynomials.

On the other side, if we do not require $a-c,a-b,b-c$ to be prime to $p$ there are cases where
$N_a,N_b,N_c$ do not generate the full symmetric field. A family of cases where this happens is
related to the factorization of $T_{p^r,1}(X,Y,Z)$, for $r\geq 1$. In fact, we have
\begin{equation*}
 T_{p^r,1}(X,Y,Z) \ =\ 
    \prod_{\substack{\alpha \in \bbF_{p^r}\\\alpha \neq 0,1}}(Z-\alpha X +(\alpha-1)Y),
\end{equation*}
as we will show below together with a few other factorizations of the polynomials 
$T_{A,B}(X,Y,Z)$, for $A,B,A-B$ not all prime to $p$.

\section{A family of counterexamples}
Let $p$ be a prime $\neq 2$, and for each $\eta \in \bbF_p$ let's consider the polynomial
\begin{equation} \label{count:eq1}
 P_\eta(X) \ =\  X^2 -2\eta X +\eta.
\end{equation}
Note that a root of $P_\eta(X)$ cannot be root of $P_\kappa(X)$ for $\eta\neq
\kappa$, because the equation
\begin{equation*}
  X^2 -2\eta X +\eta \ =\ 0,
\end{equation*}
considered as an equation in $\eta$ for a given $X$ determines univocally $\eta$, unless
$X = 1/2$, which is never a solution because $P_\eta(1/2) = 1/4 \neq 0$ for each $\eta \in \bbF_p$.

Furthermore, each $P_\eta(X)$ has distinct roots unless its discriminant $4(\eta^2-\eta)$
vanishes, and this can only happen for $\eta = 0,1$.

Thus, as $\eta$ varies in $\bbF_p$ the polynomials $P_\eta(X)$ have $2p-2$ different roots
overall, and note that $2p-2 > p$ for $p \geq 3$. Consequently, since in $\bbF_p$ there are only
$p$ elements, one of these roots will belong to $\bbF_{p^2}\setminus \bbF_p$, and this implies that
at least one of the $P_\eta(X)$ is irreducible in $\bbF_p[X]$ for some $\eta \in \bbF_p$.

Let $P_\eta(X)$ be irreducible, and $\alpha,\beta \in \bbF_{p^2}\setminus
\bbF_p$ be its roots. These roots are interchanged by the Frobenius automorphism $\scF$
\begin{equation*}
 \scF : \bar\bbF \rightarrow \bar\bbF,\qquad \tau \mapsto \tau^p.
\end{equation*}
In particular, they are interchanged applying $\scF$ any odd number of times, i.e.
\begin{equation*}
 \alpha^{p^{2k+1}} \ =\  \beta,\qquad \beta^{p^{2k+1}} \ =\  \alpha
\end{equation*}
for any integer $k$.

Note also that by construction we have
\begin{equation*}
 2\alpha \beta \ =\  2\eta \ =\  \alpha+\beta.
\end{equation*}

If we now define
\begin{equation}\label{count:eq2}
  z \ =\  \alpha x + (1-\alpha)y,\qquad w \ =\  (1-\alpha)x + \alpha y,
\end{equation}
we have that for any integer $k$
\begin{align*}
  z^{p^{2k+1}+1}+&w^{p^{2k+1}+1} \ =\  z^{p^{2k+1}}\cdot z +w^{p^{2k+1}}\cdot w \\
 =\ & (\alpha x + (1-\alpha)y)^{p^{2k+1}}(\alpha x + (1-\alpha)y) \\
  & + ((1-\alpha)x + \alpha y)^{p^{2k+1}}((1-\alpha)x + \alpha y)\\
 =\ & (\beta x^{p^{2k+1}} + (1-\beta)y^{p^{2k+1}})(\alpha x + (1-\alpha)y) \\
  & + ((1-\beta)x^{p^{2k+1}} + \beta y^{p^{2k+1}})((1-\alpha)x + \alpha y) \\
 =\ & (2\alpha\beta - \alpha-\beta+1)(x^{p^{2k+1}+1}+y^{p^{2k+1}+1}) \\
  & + (\beta+\alpha-2\beta \alpha)(x^{p^{2k+1}}y+xy^{p^{2k+1}}) \\
 =\ & x^{p^{2k+1}+1}+y^{p^{2k+1}+1}.
\end{align*}

Thus, if we take $a,b,c$ equal to  $p^{2k+1}+1,p^{2\ell+1}+1,1$ for $k>\ell\geq0$, we have found an
`alternative' pair $z,w$ in \eqref{count:eq2} that satisfies the equations
\eqref{newt:eq1}, and consequently $N_a,N_b,N_c$ cannot generate the full symmetric field.

We can also calculate the degree of the symmetric field over the field generated by
$N_{p^r+1},N_{p^s+1},N_1$ for $r>s\geq 0$, because all we have to do is count the number of
$z$ that together with some $w$ satisfy the \eqref{newt:eq1}. In particular those different from
$x,y$ can be found among the roots of $T_{p^{r-s},1}(x,y,U)^{p^s}$ considered as a polynomial 
in $U$, and given the factorization of $T_{p^{r-s},1}(x,y,U)$ in linear factors we know that they 
must be of the form $z = \alpha x + (1-\alpha)y$ for $\alpha \in \bbF_{p^{r-s}}$, $\alpha \neq
1,0$.

Furthermore, $w$ is uniquely determined as $w = (1-\alpha) x + \alpha y$, and if we put $\beta =
\alpha^{p^r}$, then $\alpha,\beta$ must satisfy the condition $2\alpha\beta = \alpha+\beta$, and
are the roots of a polynomial $P_\eta(X)$ for some $\eta \in \bar{\bbF}_p$. If $\alpha=\beta$, then 
$\alpha = 0,1$ and we get $z=y$ or $z=x$ respectively, so let's consider the case
$\alpha \neq \beta$. Now, if we put $\gamma = \alpha^{p^s}$, the condition $2\alpha\gamma
= \alpha+\gamma$ must be satisfied as well, and consequently $\beta =\gamma = \alpha^{p^s}$.

We have that $\scF^s$ (i.e. $\scF$ applied $s$ times) maps $\alpha$ to $\beta$, and applied to the
coefficients of $P_\eta(X)$ we get another polynomial of the same form, $P_\kappa(X)$ say. Since
$P_\eta(\beta) = P_\kappa(\beta) = 0$, this implies that $\eta = \kappa$, and it follows that
symmetrically $\scF^s$ maps $\beta$ to $\alpha$, and leaves $\eta$ fixed. Since the same is true for
$\scF^r$, we have that $\eta$ is fixed by $\scF^m$ for $m=(r,s)$, i.e. that $\eta \in \bbF_{p^m}$,
while $\alpha$ has degree precisely $2$ over $\bbF_{p^m}$, in other words that $\alpha \in
\bbF_{p^{2m}}\setminus \bbF_{p^m}$.

We can now distinguish two cases: when $2m \nmid (r-s)$, we have that $\bbF_{p^{2m}} \cap
\bbF_{p^{r-s}} = \bbF_{p^m}$, and the only $\alpha$ allowed are $0,1$, and
$N_{p^r+1},N_{p^s+1},N_1$ consequently generate the full symmetric field.

On the other hand, when $2m \mid (r-s)$ we have that $\bbF_{p^{2m}} \subset \bbF_{p^{r-s}}$,
and all we have to do is counting the number of solutions of
\begin{equation*}
 X^2 -2\eta X + \eta = 0,\qquad X \in \bbF_{p^{2m}}\setminus \bbF_{p^m}
\end{equation*}
while $\eta$ varies in $\bbF_{p^m}$. Repeating the same computation we did at the beginning of this
section, we deduce that the total number of solutions in $\bbF_{p^{2m}}$ is precisely $2p^m-2$,
and that each $X\in \bbF_{p^m}$ is a solution for some $\eta$ except for $X=1/2$, and the number of
these bad solutions in $\bbF_{p^m}$ is precisely $p^m-1$. Adding the trivial solutions $\alpha =
0,1$ and dividing by two we obtain the degree.

In conclusion, for any $r>s\geq 1$ and $m=(r,s)$, the degree of the symmetric field $S$ over the
field generated by $N_{p^r+1},N_{p^s+1},N_1$ is
\begin{equation*}
 [S : \clN_{p^r+1,p^s+1,1}] \ =\ \left\{
    \begin{array}{cl}
                      1 & \text{if } 2m \nmid (r-s) \\
                      \frac{p^m+1}{2} & \text{if } 2m \mid (r-s)
    \end{array}
\right.
\end{equation*}

We can easily see that in characteristic $2$ we can construct an analogous family of
counterexamples considering the pair
\begin{equation*}
  z \ =\  \alpha x + (1-\alpha)y,\qquad w \ =\  (1-\alpha) x + \alpha y,
\end{equation*}
where $\alpha \in \bbF_{2^2}\setminus \bbF_2$ is a third root of the unity, but in this case
the indices $a,b,c$ must be chosen of the form $2^{2l}+1,2^{2k}+1,1$, where \emph{even} powers of
$2$ appear.

In fact, in characteristic $2$ the condition $2\alpha\beta = \alpha+\beta$ is equivalent to
$\alpha = \beta$ (and this in characteristic $\neq 2$ can never happen, unless $\alpha =
0,1$).

To calculate the degree of $S$ over $N_{2^r+1},N_{2^s+1},N_1$ let's observe that all we have to do
is count the number of elements $\alpha \in \bbF_{2^{r-s}}$ that are left fixed by $\scF^s$
and $\scF^r$, and they are precisely the elements of $\bbF_{2^m}$, for $m = (r,s)$. Dividing by two
we get the degree of the extension
\begin{equation*}
 [S : \clN_{2^r+1,2^s+1,1}] \ =\  2^{m-1}.
\end{equation*}

\subsection{Factorization of certain families of \texorpdfstring{$T(X,Y,Z)$}{T(X,Y,Z)}.}

We will now show that
\begin{equation} \label{fact:eq1}
 T_{p^r,1}(X,Y,Z) \ =\ 
    \prod_{\substack{\alpha \in \bbF_{p^r}\\\alpha \neq 0,1}}(Z-\alpha X +(\alpha-1)Y).
\end{equation}
To check the equality, calling as usual $V(X,Y,Z)$ the Vandermonde
determinant, it is enough to observe that we have
\begin{equation*}
 T_{p^r,1}(X,Y,Z)\cdot V(X,Y,Z) \ =\
 \det\begin{pmatrix}
   X^{p^r} & Y^{p^r} & Z^{p^r} \\
   X & Y & Z \\
   1 & 1 & 1
 \end{pmatrix},
\end{equation*}
and the determinant vanishes if we put $Z = \alpha X - (\alpha-1)Y$ for each $\alpha \in
\bbF_{p^r}$. We have to exclude the factors $(Z-\alpha X +(\alpha-1)Y)$ for $\alpha =
0,1$, because they are precisely the factors dividing $V(X,Y,Z)$, but the remaining $p^r-2$ factors
are factors of $T_{p^r,1}(X,Y,Z)$, which has degree precisely $p^r-2$ in $Z$. To conclude we have
just to verify that the constant factor by which they may differ is $1$, but this is obvious
considering that the two espressions appearing in the \eqref{fact:eq1} are both monic in $Z$.

Another factorization of the same flavor is the following:
\begin{equation} \label{fact:eq2}
 T_{p^{2r}-1,p^r-1}(X,Y,Z) \ =\ 
    \prod_{\substack{\alpha,\beta \in \bbF_{p^r}\\\alpha,\beta \neq 0}}(Z-\alpha X -\beta Y).
\end{equation}
In fact, when we replace $Z = \alpha X + \beta Y$, we have that for each $\alpha,\beta \in
\bbF_{p^r}$ this substitution makes
\begin{gather*}
 T_{p^{2r}-1,p^r-1}(X,Y,Z)\cdot V(X^{{p^r}-1},Y^{{p^r}-1},Z^{{p^r}-1})\cdot XYZ \ =\ \\
 \ =\  \det\begin{pmatrix}
   X^{p^{2r}} & Y^{p^{2r}} & Z^{p^{2r}} \\
   X^{p^r} & Y^{p^r} & Z^{p^r} \\
   X & Y & Z
 \end{pmatrix}
\end{gather*}
vanish, and discarding as before the factors $(Z-\alpha X +\beta Y)$ where $\alpha = 0$ or $\beta =
0$, we are left with $p^{2r}-2p^r+1$ linear factors, and this is precisely the degree
of $T_{p^{2r}-1,p^r-1}(X,Y,Z)$. To prove the equality, we must again observe that both factors are
monic in $Z$.

It can also be interesting to observe that all these substitutions $Z = \alpha X + \beta Y$,
for $\alpha,\beta \in \bbF_{p^r}$, also make the determinant
\begin{equation*}
 \det\begin{pmatrix}
   X^{p^s} & Y^{p^s} & Z^{p^s} \\
   X^{p^t} & Y^{p^t} & Z^{p^t} \\
   X & Y & Z
 \end{pmatrix}
\end{equation*}
vanish for each $s > t \geq 1$ that are both divisible by $r$.

This determinant is also a multiple of
$T_{p^s-1,p^t-1}(X,Y,Z)$, and differs from it for a factor
$V(X^{{p^m}-1},Y^{{p^m}-1},Z^{{p^m}-1})\cdot XYZ$, that can vanish after the substitution only if
we put either $\alpha$ or $\beta$ equal to $0$.

Consequently for all $s > t \geq 1$, both divisible by $r$, we have that
\begin{equation*}
  T_{p^{2r}-1,p^r-1}(X,Y,Z) \ \Big|\  T_{p^s-1,p^t-1}(X,Y,Z).
\end{equation*}
Since actually both these polynomials are polynomials in $X^{p^r-1}$, $Y^{p^r-1}$, $Z^{p^r-1}$,
we must also have the following divisibility rule:
\begin{equation*}
 T_{p^r+1,1}(X,Y,Z) \ \Big|\  T_{\frac{p^s-1}{p^r-1},\frac{p^t-1}{p^r-1}}(X,Y,Z).
\end{equation*}

\section{Irreducibility of \texorpdfstring{$T_{p^r+1,1}(X,Y,Z)$}{T(X,Y,Z)}} We will now show
that $T_{p^r+1,1}(X,Y,Z)$,
for $r\geq 1$, that we just showed to be a factor of a class of $T(X,Y,Z)$, is irreducible. Note
that such polynomials do not belong to the family of polynomials that we have showed to be
irreducible in \emph{Case 2} of Prop. \ref{newt:prop3} proving that the projective
variety that they define is nonsingular, and in fact the point with homogeneous coordinates
$(t,t,t)$ for $t\neq 0$ is a singular point for $T_{p^r+1,1}(X,Y,Z)$.

We will use the following strategy: let
\begin{equation*}
 f(Z) \ =\  f_k Z^k + \dots + f_1 Z + f_0 \ \in k[X_1,\dots,X_n,Z]
\end{equation*}
be an homogeneous polynomial in $X_1,\dots,X_n,Z$, considered as a polynomial in $Z$ with
coefficients in $k[X_1,\dots,X_n]$. Let's suppose that there exists a $P \in
k[X_1,\dots,X_n]$ such that $f_0$ is a power of $P$, and that $P \nmid f_1$. Then $f(Z)$ is
irreducible in $k[X_1,\dots,X_n,Z]$.

Suppose in fact that $f(Z) = a(Z)b(Z)$, with $a(Z) = \sum_i{a_iZ^i}$, $b(Z)
= \sum_i{b_iZ^i}$, both of degree $\geq 1$ in $Z$. The we have that
\begin{equation*}
 f_0 \ =\  a_0 b_0,\qquad f_1 \ =\ a_1b_0 + a_0b_1.
\end{equation*}

Since the factorization is not trivial, and the factors are homogeneus, $a_0$ and $b_0$ have to be
non trivial powers of $P$, but this is absurd since it would imply that $P \mid f_1$. Consequently
$f(Z)$ cannot factored in factors with degree $\geq 1$ in $Z$, and to deduce the irreducibility in 
$k[X_1,\dots,X_n,Z]$ it suffices to show that it is primitive as a polynomial in $Z$, but this is
obvious considering that $P \nmid f_1$.

To apply this strategy to $T_{p^r+1,1}(X,Y,Z)$, consider that it is the sum of all monomials of
degree $p^r-1$, and if viewed as polynomials in $Z$ its constant term is
\begin{equation*}
 \sum_{i=0}^{p^r-1} X^iY^{p^r-1-i} \ =\ \frac{X^{p^r}-Y^{p^r}}{X-Y} \ =\ (X-Y)^{p^r-1}.
\end{equation*}
On the other hand, the coefficient of the term of degree one in $Z$ is
\begin{equation*}
 \sum_{i=0}^{p^r-2} X^iY^{p^r-2-i} \ =\ \frac{X^{p^r-1}-Y^{p^r-1}}{X-Y}
  \ =\ \prod_{\substack{\zeta^{p^r-1}=1\\\zeta \neq 1}} (X-\zeta Y).
\end{equation*}
Thus, we have verified that the constant term is a power of $X-Y$, and this is a factor that does
not appear in the coefficient of $Z$.

Consequently the irreducibility follows, applying the above
strategy. Knowing the irreducibility of this family of polynomials provides us one more case
where Proposition \ref{newt:prop3} is true, in particular when $a,b,c$ are coprime integers,
prime to the characteristic $p$, and such that $a-b = p^r$ for some $r \geq 1$ and $b-c=1$.

\bibliography{bibliografia}{}
\bibliographystyle{amsalpha}

\end{document}